\documentclass [11pt,a4paper]{article}
\usepackage[ansinew]{inputenc}
\usepackage{amssymb}
\usepackage{amsmath}
\usepackage{graphicx}
\usepackage{amsthm}
\usepackage{cite}
\newtheorem{theorem}{\textbf{Theorem}}[section]
\newtheorem{lemma}{\textbf{Lemma}}[section]
\newtheorem{proposition}{\textbf{Proposition}}[section]
\newtheorem{corollary}{\textbf{Corollary}}[section]
\newtheorem{remark}{\textbf{Remark}}[section]
\newtheorem{definition}{\textbf{Definition}}[section]

\allowdisplaybreaks[4]

\def\be{\begin{equation}}
\def\ee{\end{equation}}
\def\bea{\begin{eqnarray}}
\def\eea{\end{eqnarray}}
\def\bt{\begin{theorem}}
\def\et{\end{theorem}}
\def\bl{\begin{lemma}}
\def\el{\end{lemma}}
\def\br{\begin{remark}}
\def\er{\end{remark}}
\def\bp{\begin{proposition}}
\def\ep{\end{proposition}}
\def\bc{\begin{corollary}}
\def\ec{\end{corollary}}
\def\bd{\begin{definition}}
\def\ed{\end{definition}}
\def\non{\nonumber }

\begin{document}

\title{Finite--dimensional global attractor for a system \\modeling the
$2D$ nematic liquid crystal flow}

\author{
{\sc Maurizio Grasselli}\\
Dipartimento di Matematica,
Politecnico di Milano\\
Milano 20133, Italy\\ {\it
maurizio.grasselli@polimi.it}\\
and\\
{\sc Hao Wu}\\ Shanghai Key Laboratory for Contemporary Applied
Mathematics \\
School of Mathematical
Sciences, Fudan University\\Shanghai 200433, China\\
\textit{haowufd@yahoo.com}}

\date{\today}

\maketitle

%%%%%%%%%%%%%%%%%%%%%%%%%%%%%%%%%%%%%%%%%%%%%%%%%%%%%%%%%%%%%%%%%%%%%%%%%%%%%%%%%%%%%%%%%%%%%%%%%%%%%%%%%%%%%%%%%%%%%%%%%%%

\begin{abstract}
We consider a $2D$ system that models the nematic liquid crystal
flow through the Navier--Stokes equations suitably coupled with a
transport-reaction-diffusion equation for the averaged molecular
orientations. This system has been proposed as a reasonable
approximation of the well-known Ericksen--Leslie system. Taking
advantage of previous well-posedness results and proving suitable
dissipative estimates, here we show that the system endowed with
periodic boundary conditions is a dissipative dynamical system with
a smooth global attractor of finite fractal dimension.
\medskip

\end{abstract}

%%%%%%%%%%%%%%%%%%%%%%%%%%%%%%%%%%%%%%%%%%%%%%%%%%%%%%%%%%%%%%%%%%%%%%%%%%%%
\section{Introduction}
\label{intro}
We consider the following hydrodynamical system that models the
nematic liquid crystal flows (cf. e.g., \cite{LLZ07, SL08, dG, E61})
 \bea
 && v_t+v\cdot\nabla v-\nu \Delta v+\nabla P\non\\
 && \ =-\lambda
 \nabla\cdot[\nabla d\odot\nabla d+\alpha(\Delta d-f(d))\otimes d-(1-\alpha)d\otimes (\Delta d-f(d)) ],\label{1}\\
 && \nabla \cdot v = 0,\label{2}\\
 && d_t+(v\cdot\nabla) d-\alpha  (\nabla v) d+(1-\alpha)(\nabla^T v) d =\gamma(\Delta d-f(d)),\label{3}
 \eea
in $Q\times(0,\infty)$. Here, $Q$ is a unit square in $\mathbb{R}^2$
(the more general case $Q=\Pi_{i=1}^2(0,L_i)$ with different periods
$L_i$ in different directions can be treated in a similar way). The
state variables $v$, $d$ and $P$ represent, respectively, the
velocity field of the flow, the averaged macroscopic/continuum
molecular orientations in $\mathbb{R}^2$ and the hydrodynamic
pressure. The positive constants $\nu, \lambda$ and $\gamma$ stand
for viscosity, the competition between kinetic energy and potential
energy, and macroscopic elastic relaxation time (Deborah number) for
the molecular orientation field, respectively. The parameter
$\alpha\in [0,1]$ is related to the shape of the liquid crystal
molecule. The symbol $\nabla d\odot \nabla d$ denotes the $2\times
2$ matrix whose $(i,j)$-th entry is given by $\nabla_i d\cdot
\nabla_j d$, for $1\leq i,j\leq 2$. $\otimes$ is the usual Kronecker
product, e.g., $(a\otimes b)_{ij}=a_ib_j$ for $a,b \in
\mathbb{R}^2$. $f(d)=\frac{1}{\eta^2}(|d|^2-1)d: \mathbb{R}^2\mapsto
\mathbb{R}^2$ with $\eta\in (0, 1]$ can be seen as a penalty
function to approximate the strict unit-length constraint $|d| = 1$,
which is due to liquid crystal molecules being of similar size (cf.
\cite{LL95}). This approximation fits well with the general theory
of Landau's order parameter (cf. \cite{Le79}) and the
Ginzburg--Landau type energy is also consistent with the model on
variable degree of orientation (cf. \cite{E91}). It is obvious that
$f(d)$ is the gradient of the scalar valued function
$F(d)=\frac{1}{4\eta^2}(|d|^2-1)^2:\mathbb{R}^2\mapsto \mathbb{R}$.

In the present paper, we consider system \eqref{1}--\eqref{3}
subject to the periodic boundary conditions
 \be
v(x+e_i)=v(x),\quad d(x+e_i)= d(x),\qquad \text{for}\ x\in
\mathbb{R}^2,
 \label{4}
 \ee
 where unit vectors $e_i \ (i=1,2)$ are the canonical basis of
 $\mathbb{R}^2$. Namely, $v,d$ are well
defined in the $2$-dimensional torus
$\mathbb{T}^2=\mathbb{R}^2/\mathbb{Z}$. Besides, we have the initial conditions
 \be
 v|_{t=0}=v_0(x) \ \ \text{with}\ \nabla\cdot v_0=0,\quad
 d|_{t=0}=d_0(x),\qquad \text{for}\ x\in Q.\label{5}
 \ee
Well-posedness issues for problem \eqref{1}--\eqref{5} for
$\alpha\in [0,1]$ have been studied in \cite{WXL1} (see also
\cite{SL08} for the case $\alpha=1$ and \cite{LL01, WXL2} for the
general Ericksen--Leslie model). More recently, the existence of a
global weak solution has been proven in \cite{CR} with boundary
conditions which are not necessarily periodic. Prior to these results, a
number of papers (see, e.g., \cite{CGR06,
CGM09,FO09,GRR09,HW09,LS01,LWa00,LWa02}) have been devoted
to the theoretical and numerical analysis of the highly simplified
system studied first in \cite{LL95}. In this case, the liquid crystal
molecules are assumed to be "small" enough so that kinematic
transport is neglected. However, such an assumption is physically
questionable. On the contrary, system \eqref{1}--\eqref{3} accounts
for the kinematic transport and also preserves dissipative properties
expressed by a \emph{basic energy law} similar to \cite{LL95}
(compare with \cite{CS}). Indeed, letting $(v, d)$ be a classical
solution to problem \eqref{1}--\eqref{5}. Multiplying equation
\eqref{1} with $v$, equation \eqref{3} with $-\lambda(\Delta
d-f(d))$, adding them together and integrating over $Q$, we get (cf.
also \cite{LLZ07})
\[\frac{1}{2}\frac{d}{dt}\int_{Q}\left(|v|^2+\lambda|\nabla d|^2+2\lambda F(d)\right) dx
=-\int_{Q}\left(\nu|\nabla v|^2+\lambda\gamma|\Delta d-f(d)|^2
\right)dx. \] Taking advantage of this dissipative feature, in
\cite{WXL1} it has also been proven that a given solution converges
to a single stationary state and an estimate of the convergence rate
has been obtained (cf. \cite{LW08} for the simplified model). Here
we want to show that the dissipative dynamical system associated
with problem \eqref{1}--\eqref{5} possesses a global attractor with
finite fractal dimension (see \cite{B,S01} for the simplified
model). Our argument is slightly nonstandard since we do not know
whether the semigroup defined through the global solution of problem
\eqref{1}--\eqref{5} is strongly continuous on the phase space.
Thus, we achieve our goal by observing that the semigroup is closed
in the sense of \cite{PZ07}.

The plan of the paper goes as follows. In the next section we introduce the
functional setup, we recall the well-posedness results established
in \cite{WXL1} and we state the main theorem. Section \ref{diss} is
devoted to prove a number of dissipative estimates that entail the
existence of smooth compact absorbing sets in the phase space. This
will yield the existence of the global attractor. Finally, in
Section \ref{attr}, we prove the finite dimensionality of the global
attractor.

%%%%%%%%%%%%%%%%%%%%%%%%%%%%%%%%%%%%%%%%%%%%%%%%%%%%%%%%%%%%%%%%%%%%%%%

\section{Preliminaries and Main Result}
\label{prel}
\setcounter{equation}{0}
 We recall the well-established functional
setting for periodic boundary value problems (cf. e.g.,
\cite[Chapter 2]{Te}, also \cite{SL08}):
 \bea
H^m_p(Q)&=&\left\{v\in H^m(\mathbb{R}^2,\mathbb{R}^2)\ |\ v(x+e_i)=v(x)\right\},\non\\
 \dot{H}^m_p(Q) &=& H^m_p(Q)\cap \left\{v:\ \int_Qv(x)dx=0\
\right\},\non\\
H&=&\left\{v\in L^2_p(Q,\mathbb{R}^2) ,\ \nabla\cdot v=0\right\},\ \
\text{where}\ L^2_p(Q,\mathbb{R}^2)=H^0_p(Q),
\non\\
V&=&\{v\in \dot{H}^1_p(Q),\ \nabla\cdot v=0\},\non\\
V'&=&\text{the\ dual space of\ } V.\non
 \eea
 For the sake of simplicity, we denote the inner product on
$L^2_p(Q,\mathbb{R}^2)$ as well as  $H$ by $(\cdot,\cdot)$ and the
associated norm by $\|\cdot\|$. The space $H^m(Q,\mathbb{R}^2)$ will be shorthanded by $H^m$ and the $H^m$-inner
product ($m\in \mathbb{N}$) can be given by $ \langle v,
u\rangle_{H^m}=\sum_{|\kappa|=0}^m(D^\kappa v, D^\kappa u)$, where
$\kappa=(\kappa_1,..., \kappa_n)$ is a multi-index of length
$|\kappa|=\sum_{i=1}^n\kappa_i$ and
$D^\kappa=\partial_{x_1}^{\kappa_1},...,\partial_{x_n}^{\kappa_n}$.
For any $m\in \mathbb{N}$, $m\geq 2$, we recall
 the interior elliptic estimate, which states that for any $U_1\subset\subset U_2$
 there is a constant $C>0$ depending only on $U_1$ and $U_2$ such that
 $\|d\|_{H^m(U_1)}\leq C(\|\Delta d\|_{H^{m-2}(U_2)}+\|d\|_{L^2(U_2)})$. In our
 case, we can choose $Q'$ to be the union of $Q$ and its
 neighborhood copies. Then we have
 \be
 \|d\|_{H^m(Q)}\leq C\left(\|\Delta d\|_{H^{m-2}(Q')}+\|d\|_{L^2(Q')}\right)= 9C\left(\|\Delta
 d\|_{H^{m-2}(Q)}+\|d\|_{L^2(Q)}\right).\label{dh2}
 \ee
 Following \cite{Te}, one can define mapping $S$ (Stokes operator in
 the periodic case)
 \be S u=-\Delta u, \quad  \forall\ u\in D(S):=\{u\in H, \Delta u\in H\}=\dot H^2_p\cap
 H.\label{stokes}
 \ee
 The operator $S$ can be seen as an unbounded
positive linear self-adjoint operator on $H$. If $D(S)$ is endowed
with the norm induced by $\dot H^0_p(Q)$, then $S$ becomes an
isomorphism from $D(S)$ onto $H$. More detailed properties of
operator $S$ can be found in \cite{Te}.

We shall denote by $C$ the genetic constants depending on $\lambda,
\gamma, \nu, Q, f$ and the initial data. Special dependence will be
pointed out explicitly in the text if necessary. Since the
parameters $\nu$, $\lambda$ and $\gamma$ do not play important roles
in the proof when $n=2$, we just set $\nu=\lambda=\gamma=1$ in the
remaining part of the paper.

We now report the global existence of strong/classical solutions to
problem \eqref{1}--\eqref{5} for $\alpha\in [0,1]$ proven in
\cite{WXL1} (see \cite{SL08} for the special case $\alpha=1$).
 \bp\label{glo}
Assume $n=2$. For any initial data $(v_0, d_0)\in V\times H^2_p(Q)$,
 problem \eqref{1}--\eqref{5} admits a unique global
 solution such that
 \be v\in L^\infty(0, \infty; V)\cap L^2_{loc}(0,\infty; D(S)),\quad d\in
 L^\infty(0,\infty; H^2)\cap L^2_{loc}(0,\infty; H^3),\label{so}
 \ee
 Moreover, $v(t), d(t)\in C^\infty(Q)$ for all
 $t>0$.
 \ep
In \cite{WXL1}, the following continuous dependence result on the
initial data has also been obtained.
 \bl \label{cd}
 Suppose that $(v_i,d_i)$ are global solutions to the problem
 \eqref{1}--\eqref{5} corresponding to the initial data $(v_{0i}, d_{0i})\in V\times
 H^2_p(Q)$, $i=1,2$, respectively. Moreover, assume that the following estimate
 holds for any $T>0$
 \be
 \|v_i(t)\|_{V}+\|d_i(t)\|_{H^2}\leq M ,\quad \forall \,
 t\in [0,T].\label{BD}
 \ee
 Then for any $ t\in [0,T]$, we have
\begin{eqnarray}
&& \| (v_1-v_2)(t)\|^2+\|(d_1-d_2)(t)\|_{H^1}^2 +
\int_0^t(\|\nabla
 (v_1-v_2)(\tau)\|^2+\|\Delta( d_1-d_2)(\tau)\|^2) d\tau \non\\
 &\leq&  2e^{Ct}\left(\|v_{01}-v_{02}\|^2+\|d_{01}-d_{02}\|_{H^1}^2\right),
 \label{Lip}
\end{eqnarray}
where $C$ is a constant depending on $M$ but not on $t$.
 \el
Therefore, problem \eqref{1}--\eqref{5} has a unique strong
solution.  On account of the stated results, we are able to define a
semigroup $\Sigma(t)$ by setting $(v(t),d(t))=\Sigma(t)(v_0, d_0)$,
for all $t\geq 0$ and for any $(v_0, d_0)\in V\times H^2_p$, where
$(v,d)$ is the solution to \eqref{1}-\eqref{5}.

 \br We are not able to prove that $\Sigma(t)$ is Lipschitz continuous from $V\times
H^2_p$ to $V\times H^2_p$. However, thanks to \eqref{Lip},
$\Sigma(t)$ turns out to be a closed semigroup in the sense of
\cite{PZ07}.
 \er

 The dynamical system $(V\times H^2_p, \Sigma(t))$
is a gradient system since it has a global Lyapunov functional
 \be \mathcal{E}(t)=\frac{1}{2}\|v(t)\|^2+\frac{1}{2}\|\nabla
 d(t)\|^2+ \int_Q F(d(t))dx,\label{Ly}
 \ee
 which satisfies the following \textit{basic energy law} (cf.
 \cite{SL08,LLZ07})
 \be
 \frac{d}{dt}\mathcal{E}(t)=- \|\nabla v(t)\|^2- \|\Delta
 d(t)-f(d(t))\|^2,
 \quad \forall \, t\geq 0.\label{ENL}
 \ee

The main result of this paper is as follows:

 \bt \label{main} $(V\times H^2_p, \Sigma(t))$  possesses a connected
global attractor $\mathcal{A}$ with finite fractal dimension that is bounded in $(V\cap H_p^s)\times H_p^{s+1}$ for any given $s\in \mathbb{N}$,
$s\geq 2$.
 \et

\br Due to the existence of global Lyapunov functional
$\mathcal{E}$, a well-known result entails that $\mathcal{A}$
coincides with the unstable manifold of the set of equilibria (cf.
e.g., \cite{Z04}) . Moreover, on account of Proposition  \ref{glo}
(see also Proposition \ref{compabs}), we have that
$\mathcal{A}\subset (C^\infty(Q))^2$. \er

%%%%%%%%%%%%%%%%%%%%%%%%%%%%%%%%%%%%%%%%%%%%%%%%%%%%%%%%
\section{Dissipative Estimates}
\label{diss}
\setcounter{equation}{0}

We begin to prove the first basic dissipative inequality that is a direct consequence of the \emph{basic energy law} \eqref{ENL}.

\bl \label{bdE1} There exist constants $C_0>0, \kappa>0$
independent of initial data $(v_0, d_0)$ such that \be
 \frac{d}{dt}\mathcal{E}(t)+\kappa\mathcal{E}(t)\leq C_0, \quad \forall \, t\geq 0.\label{DE1}
 \ee
\el
\begin{proof}
Set
 \be
 g=-\Delta d+f(d)\label{eq1}.
 \ee
Multiplying \eqref{eq1} by $d$, integrating over $Q$ and using the periodic boundary condition, we get
 \be
 \frac12\|\nabla d\|^2+\int_Q(|d|^4-|d|^2)dx=\int_Q g\cdot d\,
 dx.\non
 \ee
By the Young inequality, we have
 \be
 \int_Q |d|^2 dx\leq \frac13 \int_Q |d|^4 dx+ \frac34|Q|.\label{you1}
 \ee
Then it follows that
\be
\frac12\|\nabla d\|^2+\int_Q |d|^4dx\leq \frac12\|g\|^2+\frac32\|d\|^2\leq \frac12\int_Q|d|^4dx +\frac12\|g\|^2
+ \frac98 |Q|. \label{you2}
\ee
Thus, we obtain
 \be
  \frac12\|\nabla d\|^2+\int_QF(d)dx\leq \|\nabla d\|^2+\int_Q |d|^4dx\leq \frac94|Q|+\|-\Delta
  d+f(d)\|^2.\non
  \ee
On the other hand, we have the Poincar\'e  inequality for $v\in V$ that
$\|v\|\leq C_P\|\nabla v\|$. As a result, we can see that there exist
constants $C_0>0, \kappa>0$ independent of initial data $(v_0,
d_0)$ such that
 \be
 \kappa \mathcal{E}(t)\leq \|\nabla v\|+ \|-\Delta d+f(d)\|^2+C_0,\non
 \ee
 which together with \eqref{ENL} yields \eqref{DE1}.
\end{proof}

For any $R>0$ and $(v_0, d_0)\in V\times H^2_p$ satisfying
 \be
 \|v_0\|^2_{H^1}+\|d_0\|^2_{H^2}\leq R,\non
 \ee
it is easy to infer from Lemma \ref{bdE1} the following
 \begin{proposition}
 There exists a time $t_0=t_0(R)$ and positive constants $M_1, M_2$ independent of $R$ such that
 \be
  \|v(t)\|^2+\|d(t)\|^2_{H^1} \leq M_1, \quad \forall\, t\geq t_0, \label{low1}
 \ee
 and
 \be
 \int_{t}^{t+1} \left(\|v(\tau)\|_{H^1}^2+\|d(\tau)\|_{H^2}^2\right) d\tau \leq M_2, \quad \forall\, t\geq t_0.\label{low2}
 \ee
 \end{proposition}
 \begin{proof}
 It easily follows from \eqref{DE1} that
 \be
 \mathcal{E}(t)\leq \mathcal{E}(0)e^{-\kappa t}+\frac{C_0}{\kappa}, \quad \forall\, t\geq
 t_0.\non
 \ee
 Taking
 $$t_0:=\frac{1}{\kappa}|\ln \mathcal{E}(0)+\ln \kappa-\ln C_0|,$$
 we have for all $t\geq t_0$,
 \be
 \mathcal{E}(t)\leq \frac{2C_0}{\kappa}.\non
 \ee
 This and \eqref{you1}, \eqref{you2} implies that there exists $M_1$ independent of $R$ such that \eqref{low1} holds.
 Next, from \eqref{ENL} we see that
 \be
 \int_{t}^{t+1} (\|\nabla v(\tau)\|^2+\|-\Delta d(\tau)
 + f(d(\tau))\|^2) d\tau \leq \mathcal{E}(t)\leq  \frac{2C_0}{\kappa},\quad \forall\, t\geq t_0.\label{intA}
 \ee
 Using \eqref{low1} and the following inequality
 \be
 \|d\|_{H^2}\leq C(\|\Delta d\|+\|d\|)\leq C\left(\|-\Delta d+f(d)\|+ \|d\|_{H^1}^3+\|d\|_{H^1}\right), \label{dH2}
 \ee
 we obtain the existence of $M_2$ such that \eqref{low2} holds.
 \end{proof}
 Let us now set
 \be
 A(t)=\|\nabla v(t)\|^2+ \|\Delta d(t)-f(d(t))\|^2.\label{A}
 \ee
 In \cite{WXL1} the authors proved a higher-order differential inequality uniform in $\alpha\in [0,1]$, namely,
 \bl \label{he2dd} The following inequality holds for the
 classical solution $(v,d)$ to problem \eqref{1}--\eqref{5}:
 \be
 \frac{d}{dt}A(t)+\|\Delta v(t)\|^2+\|\nabla(\Delta d(t)-f(d(t)))\|^2\leq C \,[A^2(t)+A(t)], \quad \forall
 \, t\geq 0,\label{he}
 \ee
 where $C$ is a constant depending on $\|v_0\|, \|d_0 \|_{H^1}$ and is independent of $\alpha$.
 \el
 Using \eqref{low1}, we can immediately deduce
 \bl \label{ahe2d} The inequality \eqref{he} holds for the
 classical solution $(v,d)$ to problem \eqref{1}--\eqref{5} for $t\geq t_0$, with $C$ being a constant that only depends on $M_1$.
 \el
 As a result, we have
 \begin{proposition}\label{abs1}
 There exists positive constants $M_3, M_4$ independent of $R$ such that for all $t_1:=t_0+1$, the following uniform estimates hold
 \bea
 && \|v(t)\|_{H^1}^2+\|d(t)\|^2_{H^2} \leq M_3, \quad \forall\, t\geq t_1, \label{high1}
 \\
 &&
 \int_{t}^{t+1} \left(\|v(\tau)\|_{H^2}^2+\|d(\tau)\|_{H^3}^2\right) d\tau \leq M_4, \quad \forall\, t\geq t_1.\label{high2}
 \eea
 \end{proposition}
 \begin{proof}
 It follows from Lemma \ref{ahe2d}, \eqref{intA} and the uniform Gronwall inequality that
 \be A(t+1)\leq \frac{4C_0}{\kappa}e^{\frac{2C_0}{\kappa}},\quad \forall\ t\geq t_0.\label{bdA}
 \ee
 Then by \eqref{low1} and \eqref{dH2}, there exists $M_3$ independent of $R$ such that \eqref{high1} holds.
 Besides, we infer from Lemma \ref{ahe2d} and \eqref{bdA} that
 \bea
 && \int_{t}^{t+1} (\|\Delta v(\tau)\|^2+\|\nabla(\Delta d(\tau)-f(d(\tau)))\|^2) d\tau\non\\
 &\leq&  A(t)+ C\sup_{s\in [t,t+1]}(\|A(s)\|^2+  \|A(s)\|)
 \non\\
 &\leq& M_5,\quad \forall\, t\geq t_1:=t_0+1, \label{intB}
 \eea
 where $M_5$ is independent of $R$. Since
 \bea
 \|d\|_{H^3}&\leq&
  C\left(\|\nabla \Delta d\|+\|d\|_{H^2}\right)\non\\
  &\leq& C\left(\|\nabla (\Delta d-f(d))\|+\|\nabla f(d)\|+\|d\|_{H^2}\right)\non\\
  &\leq& C\left(\|\nabla (\Delta d-f(d))\|+\|d\|_{H^2}^3+\|d\|_{H^2}\right),\label{dH3}
 \eea
 we can infer from \eqref{intB} and \eqref{high1} that \eqref{high2} holds. The proof is complete.
 \end{proof}

Hence we have

\bt The dynamical system $(V\times H^2_p, \Sigma(t))$ has a bounded
absorbing set in the phase space $V\times H_p^2$. Namely, for
any given $R>0$ and each initial data $(v_0,d_0)$ in the ball
$$\mathcal{B}_R:=\left\{(v_0,d_0)\in V\times H_p^2:
\|v_0\|_V^2+\|d_0\|_{H^2}^2\leq R\right\},$$ there exists
$\mathcal{B}_0\subset V\times H_p^2$ whose radius is
independent of $R$ such that
$(v(t),d(t))=\Sigma(t)(v_0,d_0)\in \mathcal{B}_0$ for all $t\geq t_1(R)$.
 \et

We now prove some uniform higher-order estimates for the global
solution $(v, d)$. For this purpose, we will take advantage of the
following (cf., e.g., \cite{KP}):

\bl\label{AL} When $s\geq 2$, $H^s$ is a Banach algebra. Assume that
$f, g \in H^s$. Then we have $$ \|fg\|_{H^s}\leq C\left(\|f\|_{L^\infty}\|
g\|_{H^s}+\| f\|_{H^s}\|g\|_{L^\infty}\right),$$ where the constant $C$ is
independent of $f, g$.
 \el

 \bl\label{hies}
 For any $s\in \mathbb{N}$, $s\geq 2$,  the solution $(v,d)$ satisfies the inequality
 \be
 \frac{d}{dt}\left(\|v\|_{H^s}^2+\|d\|_{H^{s+1}}^2\right)+\| v\|_{H^{s+1}}^2+\|d\|_{H^{s+2}}^2\leq
 J(t), \quad \forall\, t > 0,\label{high}
 \ee
 where $J$ is a positive function only depending on the norms
 $\|d(t)\|_{H^2}$ and $\|v(t)\|_{H^1}$ as well as on the parameter $s$.
 \el
 \begin{proof}
 Taking the $H^s$ inner-product of \eqref{1} with $v$ and adding the
 $H^{s+1}$ inner-product of  \eqref{3} with $d$, we obtain
 \bea
 && \frac12\frac{d}{dt}\left(\|v\|_{H^s}^2+\|d\|_{H^{s+1}}^2\right)+\|\nabla v\|_{H^s}^2+\|\nabla
 d\|_{H^{s+1}}^2
 \non\\
 &=& -\langle v, v\cdot\nabla v\rangle_{H^s} - \langle
v, \nabla\cdot(\nabla d \odot \nabla d)\rangle_{H^s}-\alpha \langle
v, \nabla\cdot[(\Delta d-f(d)) \otimes d]\rangle_{H^s}
 \non \\&&
 +(1-\alpha)\langle v, \nabla\cdot[d \otimes (\Delta d-f(d))
 ]\rangle_{H^s}
 - \langle v\cdot\nabla d, d\rangle_{H^{s+1}} \non\\
 &&+\alpha\langle d\cdot\nabla v, d\rangle_{H^{s+1}} -(1-\alpha)\langle d\cdot\nabla^T
 v,d\rangle_{H^{s+1}}
 - \langle f(d), d \rangle_{H^{s+1}}.\label{hse}
 \eea
 Arguing as in Proposition \ref{abs1}, we can easily show that if $v_0\in
 V$ and $d_0\in H^2$, then the uniform estimates holds
 \be
 \|v(t)\|_{H^1} + \|d(t)\|_{H^2}\leq C, \quad \forall \
 t\geq 0,\non
 \ee
 where
 $C>0$ depends on $\|v_0\|_{H^1}$ and $\|d_0\|_{H^2}$.

 Using Agmon's inequality and suitable interpolation inequalities, we have
 \bea
 \|d\|_{L^\infty} &\leq& C\|d\|^\frac12_{H^2}\|d\|^\frac12,\non\\
 \|\nabla d\|_{L^\infty}&\leq&
 C\|d\|^\frac12_{H^3}\|d\|_{H^1}^\frac12\leq C\|d\|^\frac{1}{2s}_{H^{s+2}}
 \|d\|_{H^2}^\frac{s-1}{2s}\|d\|_{H^1}^\frac12,\non\\
 \|\Delta d\|_{L^\infty}&\leq&
 C\|d\|^\frac12_{H^4}\|d\|_{H^2}^\frac12\leq C\|d\|^\frac{1}{s}_{H^{s+2}}
 \|d\|_{H^2}^\frac{s-1}{s},\non\\
 \|v\|_{L^\infty} &\leq& C\|v\|_{H^2}^\frac12\|v\|^\frac12\leq C\|v\|_{H^{s+1}}^\frac{1}{2s}\|v\|_{H^1}^\frac{s-1}{2s} \|v\|^\frac12
 \non\\
 \|\nabla v\|_{L^\infty}&\leq&
 C\|v\|^\frac12_{H^3}\|v\|_{H^1}^\frac12\leq
 C\|v\|^\frac1s_{H^{s+1}}
 \|v\|_{H^1}^\frac{s-1}{s},\non
 \\
 \|v\|_{H^s}&\leq& C\|v\|_{H^{s+1}}^\frac{s-1}{s}\|v\|_{H^1}^\frac{1}s,\non\\
 \|d\|_{H^{s+k}}&\leq&
 C\|d\|_{H^{s+2}}^\frac{s-2+k}{s}\|d\|_{H^2}^\frac{2-k}s, \quad k=1,0,-1.\non
 \eea
 Then we can apply Lemma \ref{AL} to deduce
 \bea
 && |\langle v, v\cdot\nabla v\rangle_{H^s}|\non\\
 &\leq&
 C\left(\|v\|_{L^\infty}^2\|\nabla v\|_{H^s}+\|v\|_{L^\infty}\|\nabla
 v\|_{L^\infty}\|v\|_{H^s}\right)\non\\
 &\leq&
 C\|v\|_{H^{s+1}}^\frac{s+1}{s}\|v\|_{H^1}^\frac{s-1}{s}\|v\|+C\|v\|_{H^{s+1}}^\frac{2s+1}{2s}\|v\|_{H^1}^\frac{3s-1}{2s}\|v\|^\frac12
 \non\\
  &\leq& \varepsilon\|v\|^2_{H^{s+1}}+
  C_\varepsilon\left(\|v\|_{H^1}^2\|v\|^\frac{2s}{s-1}+\|v\|_{H^1}^\frac{6s-2}{2s-1}
  \|v\|^\frac{4s}{2s-1}\right),\non
 \eea
 %%%%%%%%%%%%%%%%%%%%%%%%%%%%%%%%%%%%%%%
  \bea
 && |\langle
v, \nabla\cdot(\nabla d \odot \nabla d)\rangle_{H^s}|=|\langle
\nabla v, \nabla d \odot \nabla d\rangle_{H^s}|\non\\
 &\leq& C\|\nabla v\|_{L^\infty}\|\nabla d\|_{L^\infty}\|\nabla
 d\|_{H^s}+C\|\nabla v\|_{H^s}\|\nabla d\|_{L^\infty}^2 \non\\
 &\leq&
 C\|v\|_{H^{s+1}}^\frac1s\|d\|_{H^{s+2}}^\frac{2s-1}{2s}\|v\|_{H^1}^\frac{s-1}{s}\|d\|_{H^2}^\frac{s+1}{2s}\|d\|_{H^1}^\frac12
 + C\|v\|_{H^{s+1}}\|d\|^\frac{1}{s}_{H^{s+2}}
 \|d\|_{H^2}^\frac{s-1}{s}\|d\|_{H^1}
 \non\\
 &\leq&
\varepsilon\|v\|_{H^{s+1}}^2+\varepsilon\|d\|_{H^{s+2}}^2
+C_\varepsilon\left(
\|v\|_{H^1}^\frac{4(s-1)}{2s+1}\|d\|_{H^2}^\frac{2s+2}{2s-1}\|d\|_{H^1}^\frac{2s}{2s+1}+
 \|d\|_{H^2}^2\|d\|_{H^1}^\frac{2s}{s-1}\right),\non
 \eea
 %%%%%%%%%%%%%%%%%%%%%%%%%%%%%%%%%%%%%%%%%
 \bea
 && |\langle v\cdot\nabla d, d\rangle_{H^{s+1}}|\non\\
 &\leq& C\|v\|_{L^\infty}\|\nabla d\|_{L^\infty}\|d\|_{H^{s+1}}+C\|v\|_{L^\infty}\|\nabla
 d\|_{H^{s+1}}\|d\|_{L^\infty}\non\\
 && +C\|v\|_{H^{s+1}}\|\nabla
 d\|_{L^\infty}\|d\|_{L^\infty}\non\\
 &\leq&
 C\|v\|_{H^{s+1}}^\frac{1}{2s}\|d\|_{H^{s+2}}^\frac{2s-1}{2s}\|v\|_{H^1}^\frac{s-1}{2s}\|v\|^\frac12
 \|d\|_{H^2}^\frac{s+1}{2s}\|d\|_{H^1}^\frac12 \non\\
 && +C\|v\|_{H^{s+1}}^\frac{1}{2s}
 \|d\|_{H^{s+2}}\|v\|_{H^1}^\frac{s-1}{2s}\|v\|^\frac12\|d\|_{H^2}^\frac12\|d\|^\frac12
 \non\\
 && +C \|v\|_{H^{s+1}} \|d\|^\frac{1}{2s}_{H^{s+2}}
 \|d\|_{H^2}^\frac{2s-1}{2s}\|d\|_{H^1}^\frac12\|d\|^\frac12
 \non\\
 &\leq&
 \varepsilon\|d\|_{H^{s+2}}^2+\varepsilon\|v\|_{H^{s+1}}^2
 +C_\varepsilon\|v\|_{H^1}^\frac{s-1}{s}\|v\|
 \|d\|_{H^2}^\frac{s+1}{s}\|d\|_{H^1}\non\\
 && +C_\varepsilon
 \|v\|_{H^1}^\frac{2s-2}{2s-1}\|v\|^\frac{2s}{2s-1}\|d\|_{H^2}^\frac{2s}{2s-1}\|d\|^\frac{2s}{2s-1}+C_\varepsilon
\|d\|_{H^2}^2\|d\|_{H^1}^\frac{2s}{2s-1}\|d\|^\frac{2s}{2s-1}, \non
 \eea
 %%%%%%%%%%%%%%%%%%%%%%%%%%%%%%%%%%%%%%%%%%%
 \bea
 && |\langle f(d), d \rangle_{H^{s+1}}|\non\\
 &\leq&
 C\left(\|d\|_{L^\infty}\|d\|_{H^{s+1}}+
 \|d\|^3_{L^\infty}\|d\|_{H^{s+1}}\right)\non\\
 &\leq&
 C\|d\|_{H^{s+2}}^\frac{s-1}{s}\|d\|_{H^2}^\frac{1}s(\|d\|_{H^2}^\frac12\|d\|^\frac12+\|d\|_{H^2}^\frac32\|d\|^\frac32)
 \non\\
 &\leq& \varepsilon\|d\|_{H^{s+2}}^2
 +
 C_\varepsilon\|d\|_{H^2}^\frac{2}{s+1}
 \left(\|d\|_{H^2}^\frac12\|d\|^\frac12+\|d\|_{H^2}^\frac32\|d\|^\frac32\right)^\frac{2s}{s+1}.\non
 \eea
 %%%%%%%%%%%%%%%%%%%%%%%%%%%%%%%%%%%%%%%%%%
It remains to estimate the other four terms involving the parameter
$\alpha$ on the right-hand side of \eqref{hse}. We notice that
 \bea
 &&-\alpha \langle
v, \nabla\cdot[(\Delta d-f(d)) \otimes d]\rangle_{H^s}
+\alpha\langle d\cdot\nabla v, d\rangle_{H^{s+1}} \non\\
 &=& \alpha \langle \nabla v, (\Delta d-f(d)) \otimes d\rangle_{H^s}+\alpha\langle d\cdot\nabla v,
 d\rangle_{H^{s+1}}\non\\
 &=& -\alpha \langle \nabla v, f(d)\otimes d\rangle_{H^s}+\alpha \langle \nabla v, \Delta d\otimes d\rangle_{H^s}
 +\alpha\langle d\cdot\nabla v, d\rangle_{H^{s+1}}. \label{tala}
 \eea
The first term is estimated as follows
 \bea
  && \alpha |\langle \nabla v, f(d)\otimes d\rangle_{H^s}|\non\\
  &\leq& C\|\nabla v\|_{L^\infty}(\| d\otimes d\|_{H^{s}}+\||d|^2d\otimes d\|_{H^s})+C\|\nabla
 v\|_{H^{s}}\|f(d)\otimes d\|_{L^\infty}\non\\
 &\leq& C\|d\|_{L^\infty}\|\nabla v\|_{L^\infty}\| d\|_{H^{s}}+  C\|d\|^3_{L^\infty}\|\nabla v\|_{L^\infty}\|d\|_{H^s}
 +C\|\nabla
 v\|_{H^{s}}(\|d\|^2_{L^\infty}+\|d\|_{L^\infty}^4)\non\\
  &:=& I_1+I_2+I_3,\non
  \eea
  where
  \bea
  I_1 &\leq & C\|d\|^\frac12_{H^2}\|d\|^\frac12\|v\|^\frac1s_{H^{s+1}}
 \|v\|_{H^1}^\frac{s-1}{s}\|d\|_{H^{s+2}}^\frac{s-2}{s}\|d\|_{H^2}^\frac{2}s\non\\
 &\leq&
  \varepsilon\|d\|_{H^{s+2}}^2+\varepsilon\|v\|_{H^{s+1}}^2+C_\varepsilon
  \|v\|_{H^1}^\frac{2(s-1)}{s+1}\|d\|_{H^2}^\frac{s+4}{s+1}\|d\|^\frac{s}{s+1},
  \non
  \eea
  \bea
  I_2  &\leq& C\|d\|^\frac32_{H^2}\|d\|^\frac32\|v\|^\frac1s_{H^{s+1}}
 \|v\|_{H^1}^\frac{s-1}{s}\|d\|_{H^{s+2}}^\frac{s-2}{s}\|d\|_{H^2}^\frac{2}s  \non\\
 &\leq& \varepsilon\|d\|_{H^{s+2}}^2+\varepsilon\|v\|_{H^{s+1}}^2+C_\varepsilon
  \|v\|_{H^1}^\frac{2(s-1)}{s+1}\|d\|_{H^2}^\frac{3s+4}{s+1}\|d\|^\frac{3s}{s+1},\non
  \eea
  \bea
  I_3&\leq&
  C\|v\|_{H^{s+1}}(\|d\|_{H^2}\|d\|+\|d\|^2_{H^2}\|d\|^2)\non\\
  &\leq&
  \varepsilon\|v\|_{H^{s+1}}^2+C_\varepsilon(\|d\|^2_{H^2}\|d\|^2+\|d\|^4_{H^2}\|d\|^4).\non
  \eea
 For the remaining two terms in \eqref{tala}, we have
 \bea
 && \alpha \langle \nabla v, \Delta d\otimes d\rangle_{H^s}
 +\alpha\langle d\cdot\nabla v, d\rangle_{H^{s+1}}\non\\
 &=& \alpha \langle \nabla v, \Delta d\otimes d\rangle_{H^{s-1}}+
 \alpha \sum_{|\kappa|=s} (D^\kappa \nabla v, D^\kappa(\Delta d\otimes
 d))\non\\
 && + \alpha\langle d\cdot\nabla v, d\rangle_{H^{s}}+\alpha \sum_{|\kappa|=s+1}(D^\kappa( d\cdot\nabla v), D^\kappa
 d)\non\\
 &=& \alpha \langle \nabla v, \Delta d\otimes d\rangle_{H^{s-1}}+\alpha\langle d\cdot\nabla v, d\rangle_{H^{s}}
 + \alpha \sum_{|\kappa|=s} (D^\kappa \nabla v, D^\kappa\Delta d\otimes
 d)\non\\
 && + \alpha \sum_{|\kappa|=s} \sum_{|\xi|\leq s-1,\,\kappa\geq\xi}
 C_{\kappa, \xi}
 (D^\kappa \nabla v, D^\xi\Delta d\otimes D^{\kappa-\xi} d)\non\\
 &&- \alpha \sum_{|\kappa|=s}(D^\kappa( d\cdot\nabla v),
 D^\kappa\Delta d)\non\\
 &=& \alpha \langle \nabla v, \Delta d\otimes d\rangle_{H^{s-1}}+\alpha\langle d\cdot\nabla v, d\rangle_{H^{s}}
 + \alpha \sum_{|\kappa|=s} (D^\kappa \nabla v, D^\kappa\Delta d\otimes
 d)\non\\
 && + \alpha \sum_{|\kappa|=s} \sum_{|\xi|\leq s-1, \,\kappa\geq\xi}
 C_{\kappa, \xi}
 (D^\kappa \nabla v, D^\xi\Delta d\otimes D^{\kappa-\xi} d)\non\\
 && - \alpha \sum_{|\kappa|=s}(d \cdot D^\kappa \nabla v, D^\kappa\Delta
 d)\non\\
 && - \alpha \sum_{|\kappa|=s} \sum_{|\xi|\leq s-1, \,\kappa\geq\xi }
 C_{\kappa, \xi}
 (D^{\kappa-\xi}d\cdot D^\xi \nabla v, D^\kappa\Delta
 d)\non\\
 &=& \alpha \langle \nabla v, \Delta d\otimes d\rangle_{H^{s-1}}+\alpha\langle d\cdot\nabla v, d\rangle_{H^{s}}
 \non\\
 && + \alpha \sum_{|\kappa|=s} \sum_{|\xi|\leq s-1, \,\kappa\geq\xi }
 C_{\kappa, \xi}
 (D^\kappa \nabla v, D^\xi\Delta d\otimes D^{\kappa-\xi} d)\non\\
 && - \alpha \sum_{|\kappa|=s} \sum_{|\xi|\leq s-1, \,\kappa\geq\xi }
 C_{\kappa, \xi}
 (D^{\kappa-\xi}d\cdot D^\xi \nabla v, D^\kappa\Delta
 d)\non\\
 &:=& I_4+I_5+I_6+I_7, \label{tal}
 \eea
 \bea
 I_4&\leq& C\| \nabla v\|_{L^\infty}\|\Delta d\otimes d\|_{H^{s-1}}
 +\| \nabla v\|_{H^{s-1}}\|\Delta d\otimes d\|_{L^\infty}\non\\
 &\leq& C\| \nabla v\|_{L^\infty}\|\Delta
 d\|_{L^\infty}\|d\|_{H^{s-1}}+C\| \nabla
 v\|_{L^\infty}\|d\|_{H^{s+1}}\|d\|_{L^\infty}\non\\
 && +C\|v\|_{H^{s}}\|\Delta
 d\|_{L^\infty}\|d\|_{L^\infty}\non\\
 &\leq&
 C\|v\|_{H^{s+1}}^\frac1s\|d\|_{H^{s+2}}^\frac{s-2}{s}\|v\|_{H^1}^\frac{s-1}{s}\|d\|_{H^2}^\frac{s+2}{s}
 +C\|v\|_{H^{s+1}}^\frac1s\|d\|_{H^{s+2}}^\frac{s-1}{s}\|v\|_{H^1}^\frac{s-1}{s}\|d\|_{H^2}^\frac{s+2}{2s}\|d\|^\frac12\non\\
 && +C
 \|v\|_{H^{s+1}}^\frac{s-1}s\|d\|_{H^{s+2}}^\frac{1}{s}\|v\|_{H^1}^\frac{1}{s}\|d\|_{H^2}^\frac{3s-2}{2s}\|d\|^\frac12\non\\
 &\leq& \varepsilon
 \|v\|_{H^{s+1}}^2+\varepsilon\|d\|_{H^{s+2}}^2\non\\
 && +C_\varepsilon\left(\|v\|_{H^1}^\frac{2(s-1)}{s+1}\|d\|_{H^2}^\frac{2(s+2)}{s+1}+\|v\|_{H^1}^\frac{2(s-1)}{s}\|d\|_{H^2}^\frac{s+2}{s}\|d\|
 +\|v\|_{H^1}^\frac{2}{s}\|d\|_{H^2}^\frac{3s-2}{2s}\|d\|\right),\non
 \eea
 %%%%%%%%%%%%%%%%%%%%%%%%
 \bea
 I_5&\leq& C\|d\|_{L^\infty}^2\|\nabla
 v\|_{H^s}+C\|d\|_{L^\infty}\|\nabla v\|_{L^\infty}\|d\|_{H^s}\non\\
 &\leq&
 C\|v\|_{H^{s+1}}\|d\|_{H^2}\|d\|
 +C\|v\|_{H^{s+1}}^\frac1s\|v\|_{H^1}^\frac{s-1}{s}\|d\|_{H^{s+2}}^\frac{s-2}{s}\|d\|_{H^2}^\frac{s+4}{2s}\|d\|^\frac12\non\\
 &\leq& \varepsilon
 \|v\|_{H^{s+1}}^2+\varepsilon\|d\|_{H^{s+2}}^2+
 C_\varepsilon\left(\|d\|_{H^2}^2\|d\|^2+
 \|v\|_{H^1}^\frac{2(s-1)}{s+1}\|d\|_{H^2}^\frac{s+4}{s+1}\|d\|^\frac{s}{s+1}
   \right),\non
 \eea
 %%%%%%%%%%%%%%%%%%%%%%%%
 \bea
 I_6&\leq& C  \sum_{|\kappa|=s} \sum_{|\xi|\leq s-1, \,\kappa\geq\xi }
 \|D^\kappa \nabla v\|\|D^\xi\Delta d\|_{L^4}\|D^{\kappa-\xi}
 d\|_{L^4}\non\\
 &\leq& C\|v\|_{H^{s+1}}\sum_{0\leq m\leq s-1}
 \|d\|_{W^{m+2,4}}\|d\|_{W^{s-m,4}}\non\\
 &\leq&
 C\|v\|_{H^{s+1}}\|d\|_{H^{s+2}}^\frac{s-1}{s}(\|d\|_{H^2}^\frac1s+\|d\|_{H^2})\non\\
 &\leq& \varepsilon
 \|v\|_{H^{s+1}}^2+\varepsilon\|d\|_{H^{s+2}}^2+
 C_\varepsilon\left(\|d\|_{H^2}^2+\|d\|_{H^2}^{2s}\right),\non
 \eea
 %%%%%%%%%%%%%%%%%%%%%%%%%%%%%%%
 \bea
 I_7&\leq&  C \sum_{|\kappa|=s} \sum_{|\xi|\leq s-1,\, \kappa\geq\xi} \|D^\kappa \Delta d\|\|D^\xi \nabla v\|_{L^4}\|D^{\kappa-\xi}
 d\|_{L^4}\non\\
 &\leq& C\|d\|_{H^{s+2}} \sum_{0\leq m\leq
 s-1}\|v\|_{W^{m+1,4}}\|d\|_{W^{s-m,4}}\non\\
 &\leq&
 C\sum_{0\leq m\leq
 s-2}\|d\|_{H^{s+2}}^\frac{4s-2m-3}{2s}\|v\|_{H^{s+1}}^\frac{2m+1}{2s}\|d\|_{H^2}^\frac{2m+3}{2s}\|v\|_{H^1}^\frac{2s-2m-1}{2s}\non\\
 && + C\|d\|_{H^{s+2}}\|v\|_{H^{s+1}}^\frac{s-1}{s}\|v\|_{H^1}^\frac1s\|d\|_{H^2}\non\\
 &\leq& \varepsilon \|v\|_{H^{s+1}}^2+\varepsilon\|d\|_{H^{s+2}}^2+
 C_\varepsilon\left(\|d\|_{H^2}^{2s+2}+\|v\|_{H^1}^{2s+2}+\|v\|_{H^1}^2\|d\|_{H^2}^{2s}\right).\non
 \eea
%%%%%%%%%%%%%%%%%%%%%%%%%%%%%%%%%
  The term $$(1-\alpha)\langle v, \nabla\cdot[d \otimes (\Delta d-f(d))
 ]\rangle_{H^s}
 -(1-\alpha)\langle d\cdot\nabla^T
 v,d\rangle_{H^{s+1}}$$ can be treated exactly as in \eqref{tala} and \eqref{tal}.
On the other hand, the interpolation inequalities and Young
inequality yield that
 \be
 \|v\|_{H^s}\leq \varepsilon\|v\|_{H^{s+1}}+
 C_\varepsilon\|v\|_{H^1},\quad  \|d\|_{H^{s+1}}
 \leq \varepsilon\|d\|_{H^{s+2}}+ C_\varepsilon\|d\|_{H^2}.\non
 \ee
 Collecting all the above estimates and taking $\varepsilon$
 sufficiently small, we obtain
 \be
 \frac{d}{dt}\left(\|v\|_{H^s}^2+\|d\|_{H^{s+1}}^2\right)+\| v\|_{H^{s+1}}^2+\|d\|_{H^{s+2}}^2\leq J(t),
 \ee
 where $J(t)$ depends only on $\|v(t)\|_{H^1}$, $\|d(t)\|_{H^2}$ and
 $s$.
 \end{proof}
 \br Since $\alpha \in [0, 1]$, it is easy to realize that the
higher-order differential inequality \eqref{high} is uniform in
$\alpha$.
 \er

 \begin{proposition}
 \label{compabs}
 For each $s\in \mathbb{N}$, $s\geq 2$,
 there exist positive constants
 $M_6, M_7$ independent of $R$ such that for all
 $t_s:=t_0+s$, the following uniform estimates hold
 \be
  \|v(t)\|_{H^s}^2+\|d(t)\|^2_{H^{s+1}} \leq M_6, \quad \forall\, t\geq t_s, \label{high1a}
 \ee
 and
 \be
 \int_{t}^{t+1} \left(\|v(\tau)\|_{H^{s+1}}^2+\|d(\tau)\|_{H^{s+2}}^2\right) d\tau \leq M_7,
 \quad \forall\, t\geq t_s.\label{high2a}
 \ee
 \end{proposition}
\begin{proof}
 By \eqref{high1}, we can see that inequality \eqref{high} holds for $t\geq t_1$,
 with $J(t)$ being uniformly bounded by a constant $J_0=J_0(M_1, s)$ depending only on $M_1$ and $s$.
 We argue by induction on $s$.
 For $s=2$, it follows that
 \be
 \int_{t}^{t+1} J(\tau)d\tau \leq J_0(M_1, 2),\quad \forall \, t\geq
 t_1.\non
 \ee
 We conclude from  \eqref{high2} and the uniform Gronwall
 inequality that
 \be
  \|v(t+1)\|_{H^2}^2+\|d(t+1)\|^2_{H^{3}} \leq M_4+J_0(M_1,2), \quad \forall \, t\geq
  t_1.\non
 \ee
  For $t\geq t_2:=t_1+1$, integrating \eqref{high} from $t$ to
  $t+1$, we obtain
 \be
 \int_{t}^{t+1} (\|v(\tau)\|_{H^{3}}^2+\|d(\tau)\|_{H^{4}}^2) d\tau \leq M_4+2J_0(M_1,2),
 \quad \forall\, t\geq t_2:=t_1+1.\non
 \ee
 Assume that for $s=k$, we have the following uniform estimates for
 $t_k:=t_1+k-1$:
 \bea
  && \|v(t)\|_{H^k}^2+\|d(t)\|^2_{H^{k+1}} \leq K_1, \quad \forall\, t\geq t_k, \label{high1ab}
 \\
 && \int_{t}^{t+1} \left(\|v(\tau)\|_{H^{k+1}}^2+\|d(\tau)\|_{H^{k+2}}^2\right) d\tau \leq K_2,
 \quad \forall\, t\geq t_k,
 \label{high2ab}
 \eea
 where $K_1$ and $K_2$ are  constants independent of $R$.
 Then repeating the above argument, we can see that
 for $t\geq t_{k+1}:=t_1+k$,
  \bea
 && \|v(t)\|_{H^{k+1}}^2+\|d(t)\|^2_{H^{k+2}} \leq K_2+J_0(M_1, k+1), \quad \forall \, t\geq
  t_{k+1},\non
 \\
 && \int_{t}^{t+1} \left(\|v(\tau)\|_{H^{k+2}}^2+\|d(\tau)\|_{H^{k+3}}^2\right) d\tau \leq K_2+2J_0(M_1, k+1),
 \quad \forall\, t\geq  t_{k+1}.\non
 \eea
 The proof is complete.
 \end{proof}

Proposition \ref{compabs} implies the existence of a compact
absorbing set for our dynamical system:

 \bt \label{cab}  For each $s\in \mathbb{N}$, $s\geq 2$, the dynamical system $(V\times H^2_p, \Sigma(t))$
has a compact absorbing set bounded in the space $(V\cap
H_p^s)\times H_p^{s+1}$. Namely, for any given $R>0$ and each initial
data $(v_0,d_0)$ in the ball
$$\mathcal{B}_R:=\left\{(v_0,d_0)\in V\times H_p^2:
\|v_0\|_V^2+\|d_0\|_{H^2}^2\leq R\right\},$$ there exists
$\mathcal{B}_1\subset (V\cap H_p^s)\times H_p^{s+1}$ whose radius is
independent of $R$ such that
$$(v(t),d(t))=\Sigma(t)(v_0,d_0)\in \mathcal{B}_1, \quad \text{for all}\  t\geq t_s(R).$$
 \et

\section{Finite-dimensional Global Attractor}
\label{attr}
\setcounter{equation}{0}

To prove the existence of the global attractor, we make use of the
following abstract result (cf. \cite[Corollary
6]{PZ07})
 \bl
 Let the closed semigroup $\Sigma(t)$ has a connect compact
 attracting set $\mathcal{K}$. Assume also that
 $\Sigma(t)\mathcal{K}\subset \mathcal{K}$ for every $t$ sufficient
 large. Then $\Sigma(t)$ has a connected global attractor $\mathcal{A}$.
 \el
 Owing to the above lemma, we can take $\mathcal{K}=\mathcal{B}_1$ as in Theorem
 \ref{cab} and obtain the existence of the global attractor $\mathcal{A}$. Its boundedness
 properties follow from Proposition \ref{compabs}.

 It remains to prove that the attractor $\mathcal{A}$ has finite
 fractal dimension. First, we recall the definition of box-counting dimension.
 \bd
Let $X$ be a (relatively) compact subset of a metric space $E$. For a given
$\epsilon > 0$, let $N_\epsilon(X)$ be
the minimal number of balls of radius $\epsilon$ that are necessary
to cover $X$. Denote the Kolmogorov $\epsilon$-entropy of $X$ in $E$
by $\mathcal{H}_\epsilon(X)= {\rm log}_2 N_\epsilon(X)$. Then the
fractal dimension of $X$ is the quantity ${\rm dim}_F X :=
\limsup_{\epsilon\to 0}\frac{ \mathcal{H}_\epsilon(X)}{ {\rm log}_2
\epsilon^{-1}}$.
 \ed
Also, we report a general result that ensures the finite
fractal dimensionality of a compact set, namely (cf., e.g., \cite[Theorem
4.1]{Z00}),
 \bl\label{ABSF}
  Let $X$ be a compact subset of Banach space $E$. We assume that there exist a Banach
space $E_1$ such that $E_1$ is compactly embedded into $E$ and a
mapping $L: X\to X$ such that $L(X) = X$ and
 \be \|L(x_1)-L(x_2)\|_{E_1} \leq c\|x_1- x_2\|_E, \quad \forall \ x_1, x_2\in
 X. \non
 \ee
 Then the fractal dimension of $X$ is finite and satisfies
 $${\rm dim}_F X \leq \mathcal{H}_{\frac{1}{4c}} (B_{E_1}(0, 1)),$$
 where $B_{E_1}(0, 1)$ is the unit ball at origin in
 $E_1$.
 \el

We now prove the following smoothing property:

 \bl
 For any given $(v_{0i}, d_{0i})\in \mathcal{A}$, $i=1,2$, the corresponding
 complete bounded trajectories $(v_{i},d_{i})$ satisfy the estimate
  \be
  \|(v_1-v_2)(1)\|_{H^2}^2+\|(d_1-d_2)(1)\|_{H^{3}}^2
   \leq C\left(\|v_{01}-v_{02}\|_{H^1}^2+\|d_{01} - d_{02}\|_{H^{2}}^2\right),\non
 \ee
 where $C$  is a constant depending on $\|v_{i0}\|_{H^{3}}$ and
 $\|d_{0i}\|_{H^{4}}$.
 \el
\begin{proof}
 Denote
\begin{eqnarray}
\bar v=v_1-v_2, \ \  \bar d=d_1-d_2, \ \ \bar v_0=v_{01}-v_{02}, \ \
\bar d_0=d_{01}-d_{02}.
\end{eqnarray}
Since $(v_i,d_i)$ are solutions to problem \eqref{1}--\eqref{5}, we
have
 \bea
 && v_{1t}+v_1\cdot\nabla v_1-\nu \Delta v_1+\nabla P_1\non\\
 &&\ =-
 \nabla\cdot[\nabla d_1\odot\nabla d_1+\alpha(\Delta d_1-f(d_1))\otimes d_1-(1-\alpha)d_1\otimes (\Delta d_1-f(d_1))],\label{1 for the first}\\
 && \nabla \cdot v_1 = 0,\label{2 for the first}\\
 && d_{1t}+v_1\cdot\nabla d_1-\alpha  (\nabla v_1) d_1+(1-\alpha)(\nabla^T v_1) d_1=\Delta d_1-f(d_1),\label{3 for the first}
 \\
 && v_{2t}+v_2\cdot\nabla v_2-\nu \Delta v_2+\nabla P_2\non\\
 &&\ =-\nabla\cdot[\nabla d_2\odot\nabla d_2+\alpha(\Delta d_2-f(d_2))\otimes d_2-(1-\alpha)d_2\otimes (\Delta d_2-f(d_2))],\label{1 for the second}\\
 && \nabla \cdot v_2 = 0,\label{2 for the second}\\
 && d_{2t}+v_2\cdot\nabla d_2-\alpha  (\nabla v_2) d_2+(1-\alpha)(\nabla^T v_2) d_2=\Delta d_2-f(d_2),
 \label{3 for the second}
 \eea
in $Q\times\mathbb{R}$. For $s\geq 1$, taking the $H^s$-inner product of $\bar v$ with the equation obtained
by subtracting \eqref{1 for the second} from
\eqref{1 for the first} and the $H^{s+1}$-inner product of $ \bar d$
with the equation obtained by subtracting \eqref{3 for the second} from \eqref{3 for
the first}, respectively, adding the two resulting equations together, we
obtain
\begin{eqnarray}
&& \frac{1}{2}\frac{d}{dt}\left(\|\bar v\|_{H^s}^2+\|\bar
    d\|_{H^{s+1}}^2
    \right)+\|\nabla \bar v\|_{H^s}^2+\|\nabla \bar d\|_{H^{s+1}}^2\nonumber\\
&=& -\langle v_2\cdot\nabla \bar v, \bar v\rangle_{H^s}
    -\langle\bar v\cdot \nabla v_1, \bar v\rangle_{H^s}
    -\langle \Delta \bar d\cdot \nabla d_1, \bar v\rangle_{H^s} \non\\
&&  -\langle \Delta d_2\cdot \nabla \bar d, \bar v\rangle_{H^s} 
    +\alpha \langle\Delta \bar d \otimes d_1, \nabla \bar
    v\rangle_{H^s} 
    +\alpha \langle\Delta d_2 \otimes \bar d, \nabla \bar
    v\rangle_{H^s}\non\\
&&  -\alpha\langle(f(d_1)-f(d_2))\otimes d_1, \nabla \bar
    v\rangle_{H^s}
    -\alpha\langle f(d_2)\otimes \bar d,\nabla \bar
    v\rangle_{H^s}\non\\
&&  -(1-\alpha)\langle d_1 \otimes \Delta \bar d, \nabla
    \bar v\rangle_{H^s} 
    -(1-\alpha)\langle\bar d \otimes \Delta d_2, \nabla
    \bar v\rangle_{H^s}\non\\
&&  +(1-\alpha)\langle d_1 \otimes (f(d_1)-f(d_2)),
    \nabla \bar
    v\rangle_{H^s}
    +(1-\alpha)\langle\bar d \otimes f(d_2),\nabla \bar
    v\rangle_{H^s}\non\\
&&  -\langle f(d_1)-f(d_2),  \bar d\rangle_{H^{s+1}}
    -\langle\bar v\cdot\nabla d_1 , \bar d\rangle_{H^{s+1}}
    -\langle v_2\cdot\nabla \bar d,  \bar d\rangle_{H^{s+1}}\non\\
&&  +\alpha\langle(\nabla \bar v) d_1,\bar
    d\rangle_{H^{s+1}}
    +\alpha\langle(\nabla v_2)\bar d, \bar
    d\rangle_{H^{s+1}}\non\\
&&  -(1-\alpha)\langle(\nabla^T \bar v) d_1,\bar
    d\rangle_{H^{s+1}}
    -(1-\alpha)\langle(\nabla^T v_2)\bar d, \bar
    d\rangle_{H^{s+1}}.
\label{eni}
\end{eqnarray}
 Since $\mathcal{A}\subset (C^\infty(Q))^2$, then $ (v_{0i},
d_{0i})\in H^s\times H^{s+1}$ for any $s\geq 2$. We infer from
\eqref{high} that
 \be
 \|v_i(t)\|_{H^s}^2+\|d_i(t)\|_{H^{s+1}}^2\leq
 \left(\|v_{0i}\|_{H^s}^2+\|d_{0i}\|_{H^{s+1}}^2\right)e^{-t}+J,\quad \forall\,
 t\geq 0,\non
 \ee
 where $J$ depends only on $\|v_i\|_{H^1}$, $\|d_i\|_{H^2}$ that are
 uniformly bounded by Lemma \ref{he2dd}.
 As a result, we have uniform-in-time estimates of the Sobolev
 norms of $(v_i, d_i)$ of any order $s\in \mathbb{N}$. Using these higher-order estimates, it is not
 difficult to bound the right-hand side of \eqref{eni} as in Lemma
 \ref{hies} (actually much simpler). We have
 \be
 \text{r.h.s of \eqref{eni}} \leq \varepsilon \|\bar
 v\|^2_{H^{s+1}}+\varepsilon\|\bar d\|_{H^{s+2}}^2+ C_\varepsilon \left(\|\bar
 v\|^2_{H^{s}}+\|\bar d\|_{H^{s+1}}^2\right).\non
 \ee
 Choosing  $\varepsilon=\frac12$, we obtain that
 \be
 \frac{d}{dt}\left(\|\bar v\|_{H^s}^2+\|\bar
   d\|_{H^{s+1}}^2
    \right)+\| \bar v\|_{H^{s+1}}^2+\| \bar
    d\|_{H^{s+2}}^2\leq K(s)\left(\|\bar v\|_{H^s}^2+\|\bar
   d\|_{H^{s+1}}^2\right).\label{diff}
 \ee
  where $K(s)$ is a constant depending on $\|v_{0i}\|_{H^{s+1}}$ and
 $\|d_{0i}\|_{H^{s+2}}$ at most.

 Taking $s=1$ in \eqref{diff}, it follows from the Gronwall inequality that
 \be
  \| \bar v(t)\|_{H^{1}}^2+\| \bar
    d(t)\|_{H^{2}}^2\leq e^{K(1)t} \left(\|\bar v_0\|_{H^1}^2+\|\bar
   d_0\|_{H^{2}}^2\right),\quad \forall\, t\geq 0,\non
   \ee
   which implies
   \be
  \int_0^1(\| \bar v(t)\|_{H^{2}}^2+\| \bar
    d(t)\|_{H^{3}}^2)dt\leq \left(1+e^{K(1)}\right)\left(\|\bar v_0\|_{H^1}^2+\|\bar
   d_0\|_{H^{2}}^2\right).\non
 \ee
 Next, taking $s=2$, multiplying  \eqref{diff} by $t$ and integrating in
 time from $0$ to $1$, we obtain
 \bea
 && \|\bar v(1)\|_{H^2}^2+\|\bar
   d(1)\|_{H^{3}}^2
   \non\\
   &\leq& \int_0^1\left(\| \bar v(t)\|_{H^{2}}^2+\| \bar
    d(t)\|_{H^{3}}^2\right)dt+K(2) \int_0^1 t\left(\| \bar v(t)\|_{H^{2}}^2+\| \bar
    d(t)\|_{H^{3}}^2\right)dt\non\\
    &\leq& (1+K(2))\left(1+e^{K(1)}\right)\left(\|\bar v_0\|_{H^1}^2+\|\bar
   d_0\|_{H^{2}}^2\right).\non
 \eea
The proof is complete.
\end{proof}

 Taking $$X=\mathcal{A},\quad E=V\times
 H_p^2,\quad E_1=(V\cap H_p^2)\times
 H_p^3,\quad L=\Sigma(1),$$
  we can apply Lemma \ref{ABSF} to conclude
 that the global attractor $\mathcal{A}$ has finite fractal
 dimension in the $H^1 \times H^2$-metric.
 The proof of Theorem \ref{main} is now complete.

%%%%%%%%%%%%%%%%%%%%%%%%%%%%%%%%%%%%%%%%%%%%%%%%%%%%%%%%%%%%%%%%%%%%%%%%

\bigskip
\noindent\textbf{Acknowledgments.} This work originated from a visit
of the first author to the Fudan University whose hospitality is
gratefully acknowledged. The first author has also been partially
supported by the the Italian MIUR-PRIN Research Project 2008 {\it
Transizioni di fase, isteresi e scale multiple}. The second author
was partially supported by NSF of China 11001058, NSF of Shanghai
10ZR1403800 and "Chen Guang" project supported by Shanghai
Municipal Education Commission and Shanghai Education Development
Foundation.
\bigskip
%%%%%%%%%%%%%%%%%%%%%%%%%%%%%%%%%%%%%%%%%%%%%%%%%%%%%%%%%%%%%%%%%%%%%%%%

\end{document}